\documentclass[reqno, 12pt]{amsart}
\input epsf
\usepackage[arrow,matrix]{xy}
\usepackage{graphicx}
\usepackage{amsmath}
\usepackage{hyperref}
\usepackage{comment}
\usepackage{color}
\usepackage{amsmath, latexsym, amssymb}
\usepackage{epstopdf}
\input xypic

\numberwithin{equation}{section}
\theoremstyle{plain}
\newtheorem{lemma}{Lemma}[section]
\newtheorem{proposition}[lemma]{Proposition}
\newtheorem{theorem}[lemma]{Theorem}

\theoremstyle{definition}
\newtheorem{definition}[lemma]{Definition}
\newtheorem{remark}[lemma]{Remark}
\newtheorem{example}[lemma]{Example}

\begin{document}
\newcommand{\R}{{\mathbb R}}
\newcommand{\C}{{\mathbb C}} 
\newcommand{\Ii}{{\mathcal I}}
\newcommand{\Jj}{{\mathcal J}}
\newcommand{\Nn}{{\mathcal N}}
\newcommand{\Ll}{{\mathcal L}}
\newcommand{\Tt}{{\mathcal T}}
\newcommand{\Gg}{{\mathcal G}}
\newcommand{\Dd}{{\mathcal D}}

\newcommand{\pr}{\operatorname{pr}}
\newcommand{\bla}{\langle \! \langle}
\newcommand{\bra}{\rangle \! \rangle}
\newcommand{\blq}{[ \! [}
\newcommand{\brq}{] \! ]}

\definecolor{orange}{rgb}{0,0,0}

\date{\today}

\title[Generalized Contact Bundles]
{Generalized Contact Bundles}

\author{Luca Vitagliano}
\address{DipMat, Universit\`a degli Studi di Salerno \& Istituto Nazionale di Fisica Nucleare, GC Salerno, via Giovanni Paolo II n${}^\circ$ 123, 84084 Fisciano (SA) Italy.}
\email{lvitagliano@unisa.it}

\author{A\"issa Wade}
\address{Department of Mathematics, Penn State University, University Park, State College, PA 16802, USA.}
\email{wade@math.psu.edu}

\begin{abstract} In this Note, we propose  a line bundle approach to  odd-dimensional analogues of generalized complex structures. This new approach  has three main advantages:  (1) it encompasses all existing ones;  (2) it elucidates the geometric meaning of the integrability condition for generalized contact structures; 
(3) in light of new results on multiplicative forms and  Spencer operators \cite{CSS2015}, it allows a simple interpretation  of the defining equations of a generalized contact structure in terms of Lie algebroids and Lie groupoids.
\end{abstract}



\maketitle

\section{Introduction}
Generalized complex structures have been introduced by Hitchin in \cite{H2003} and further investigated by Gualtieri in \cite{G2011}. {They can only be supported by even-dimensional manifolds and encompass symplectic structures and complex structures as extreme cases}. Since both of these extreme cases have analogues in odd-dimensional  geometry (namely, contact and almost contact structures, respectively), it is natural to ask if there is any natural odd-dimensional analogue of generalized complex structures. 
 
Several approaches to  odd-dimensional analogues of  generalized complex structures can be found in the literature \cite{IW2005,V2008,PW2011,S2015,AG2015}. They are often  named \emph{generalized contact structures} and all of them  include contact structures globally defined by a  contact 1-form. However, none of them incorporates  non-coorientable contact structures. From a conceptual point of view, contact geometry is the geometry of an hyperplane distribution and the choice of a contact form is just a technical tool making things simpler. {Even more, there are interesting contact structures that do not possess any global contact form}. Accordingly, it would be nice to define a generalized contact structure ``independently of the choice of a contact form''. This Note fills that gap. We call the proposed structure a \emph{generalized contact bundle} to distinguish it from previously defined generalized contact structures. Generalized contact bundles are just a slight generalization of Iglesias-Wade integrable generalized almost contact structures \cite{IW2005} to the realm of (generically non-trivial) line bundles. Generalized contact bundles encompass (generically non-coorientable) contact structures and complex structures on the Atiyah algebroid of a line bundle as extreme cases. This new point of view on generalized contact geometry could  also be useful in studying $T$-duality \cite{AG2015}.

{In this Note, we interpret the defining equations of a generalized contact structure in terms of Lie algebroids and Lie groupoids. As a side result we define a novel notion of multiplicative Atiyah form on a Lie groupoid and identify its infinitesimal counterpart. This could have an independent interest.}

\section{The Atiyah algebroid  associated to a contact distribution} 
For a better understanding of the concept of a generalized contact bundle, we briefly discuss a line bundle approach to contact geometry. By definition,  a contact structure on an odd-dimensional manifold $M$ is a maximally non-integrable hyperplane distribution $H \subset TM$. In a dual way, any hyperplane distribution $H$ on $M$ can be regarded as a nowhere vanishing $1$-form $\theta : TM \to L$ {(its \emph{structure form})} with values in the line bundle $L =TM/ H$, {such that $H = \ker \theta$}. Now, consider the
 so called \emph{Atiyah algebroid} $DL \to M$ (also known as \emph{gauge algebroid}  \cite{Kosmann-Mackenzie}, \cite{Chen-Liu}) of  the line bundle $L$ \cite[Sections 2, 3]{V2015}. 
Recall that sections of $D L$ are \emph{derivations} of $L$, i.e.~$\R$-linear operators $\Delta : \Gamma (L) \to \Gamma (L)$ such that there exists a, necessarily unique, vector field $\sigma\Delta \in \frak X (M)$, called the \emph{symbol of $\Delta$}, such that $\Delta (f \lambda) = (\sigma\Delta)(f) \lambda + f \Delta (\lambda)$ for all $f \in C^\infty (M)$ and $\lambda \in \Gamma (L)$. This is a transitive Lie algebroid whose Lie bracket is the commutator, and whose anchor $D L \to TM$ is the symbol $\sigma$.  Additionally, $L$ carries a \emph{tautological representation} of $D L$ given by the action of an operator on a section. Any  $k$-cochain in the de Rham complex $(\Omega^\bullet_L := \Gamma (\wedge^\bullet (D L)^\ast \otimes L), d_{D L} )$ of $D L$ with coefficients in $L$ will be called an $L$-valued \emph{Atiyah $k$-form}.
{There is a one-to-one correspondence between contact structures $H$ with $TM/H = L$ and non-degenerate, $d_{DL}$-closed, $L$-valued Atiyah $2$-forms. Contact structure $H$, with structure form $\theta : TM \to L$, corresponds to the Atiyah $2$-form $\omega := d_{DL} \sigma^\ast \theta$, where $\sigma^\ast \theta (\Delta) := \theta (\sigma \Delta)$.}

 For more details on Atiyah forms as well as their functorial properties, see \cite[Section 3]{V2015}.

\section{Generalized contact bundles and contact-Hitchin pairs}

{Recall that a generalized almost complex structure on a manifold $M$ is an endomorphism $\Jj : \mathbb T M \to \mathbb TM$ of the \emph{generalized tangent bundle $\mathbb T M := TM \oplus T^\ast M$} such that (1) $\Jj^2 = - \operatorname{id}$, and (2) $\Jj$ is skew-symmetric with respect to the natural pairing on $\mathbb T M$. If, additionally, (3) the $\sqrt{-1}$-eigenbundle of $\Jj$ in the complexification $\mathbb TM \otimes \mathbb C$  in involutive relative to the Dorfman (equivalently, the Courant) bracket, then $\Jj$ is said to be \emph{integrable}, and  $(M, \mathcal J)$ is called a \emph{generalized complex manifold} (see \cite{G2011} for more details).}

Replacing the tangent algebroid with the Atiyah algebroid of a line bundle in the definition of a generalized complex manifold, we obtain the notion of \emph{generalized contact bundle}. More precisely, let $L \to M$ be a line bundle. For a vector bundle $V \to M$, there is an obvious $L$-valued (duality) pairing $\langle -, - \rangle_L : V \otimes (V^\ast \otimes L) \to L$, and for every vector bundle morphism $F: V \to W$ (covering the identity) there is an adjoint morphism $F^\dag: W^\ast \otimes L \to V^\ast \otimes L$ uniquely determined by $\langle F^\dag (\phi), v \rangle_L = \langle \phi , F(v) \rangle_L   $, $\phi \in W^\ast \otimes L$, $v \in V$. Clearly, $(D L)^\ast \otimes L = J^1 L$, the first jet bundle of $L$. The direct sum $\mathbb D L := D L \oplus J^1 L$ is a contact-Courant algebroid (in the sense of Grabowski \cite{G2013}), and is called an \emph{omni-Lie algebroid} in \cite{Chen-Liu}. We denote by $\blq -, - \brq : \Gamma (\mathbb D L) \times \Gamma (\mathbb D L) \to \Gamma (\mathbb D L)$ and $\bla - , - \bra : \mathbb D L \otimes \mathbb D L \to L$, the Dorfman-Jacobi bracket and the $L$-valued symmetric pairing, respectively. Namely, for all $\Delta, \nabla \in \Gamma (D L)$, $\phi, \psi \in \Gamma (J^1 L)$,
\[
\bla (\Delta , \phi), (\nabla, \psi) \bra := \langle \Delta ,\psi \rangle_L + \langle \nabla , \phi \rangle_L,
\]
and
\[
\blq (\Delta, \phi), (\nabla, \psi) \brq := \left([\Delta, \nabla], \Ll_\Delta \psi - i_\nabla d_{D L}  \phi \right).
\]
See e.g.~\cite{V2015} for the main properties of these structures.

\begin{definition}\label{def:almost}
A \emph{generalized almost contact bundle} is a line bundle $L \to M$ equipped with a \emph{generalized almost contact structure}, i.e.~an endomorphism $\Ii : \mathbb D L \to \mathbb D L$ such that $\Ii^2 = - \operatorname{id}$, and $\Ii^\dag = - \Ii$.
\end{definition}

\begin{remark}\label{rem:E}
Similarly as for generalized almost complex structures, it is easy to see that, in the case $L = M \times \R$, one recovers \cite[Definition 4.1]{IW2005} which is equivalent to Sekiya's generalized $f$-almost contact structures \cite{S2015}. In particular, Poon-Wade's generalized almost contact pairs are special cases of Definition \ref{def:almost}.
\end{remark}

Using direct sum decomposition $\mathbb D L = D L \oplus J^1 L$, one sees that every generalized almost contact structure on $L$ is of the form
\begin{equation}\label{eq:split}
\Ii = \left(
\begin{array}{cc}
\varphi & J^\sharp \\
\omega^\flat & -\varphi^\dag
\end{array}
\right)
\end{equation}
where

\noindent {\bf (i)} $\varphi : D L \to D L$ is a vector bundle endomorphism, 

\noindent {\bf (ii)} $J : \wedge^2 J^1 L \to L$ is a $2$-form with associated morphism $J^\sharp : J^1 L \to D L$, and

\noindent {\bf (iii)}  $\omega : \wedge^2 D L \to L$ is a $2$-form with associated morphism $\omega^\flat : D L \to J^1 L$,

\noindent satisfying the relations:
\begin{equation}\label{eq:alg}
\varphi J^\sharp = J^\sharp \varphi^\dag;  \quad
\varphi^2  = - \operatorname{id}- J^\sharp \omega^\flat;   \quad  {\rm and} \quad 
\omega^\flat \varphi  = \varphi^\dag \omega^\flat. 
\end{equation}
Conversely, every triple $(\varphi, J, \omega)$ as above determines a generalized almost contact structure via (\ref{eq:split}).
From  the third equation in (\ref{eq:alg}), putting $\omega_\varphi (\Delta, \nabla) := \omega (\varphi \Delta, \nabla)$,
we get a well defined Atiyah $2$-form $\omega_\varphi$. 
Following \cite{IW2005} we introduce the:

\begin{definition}
A generalized almost contact structure $\Ii$ on $L$ is \emph{integrable} if its \emph{Nijenhuis torsion} $\Nn_{\Ii} : \Gamma (\mathbb D L) \times \Gamma (\mathbb D L) \to \Gamma (\mathbb D L)$, defined by
$
\Nn_{\Ii} (\alpha, \beta) := \blq \Ii \alpha, \Ii \beta \brq - \blq  \alpha, \beta \brq - \Ii \blq \Ii \alpha , \beta \brq - \Ii \blq \alpha, \Ii \beta \brq
$,
vanishes identically. A \emph{generalized contact structure} is an integrable generalized almost contact structure. A \emph{generalized contact bundle} is a line bundle equipped with a generalized contact structure.
\end{definition}

Now, a section $J \in \Gamma (\wedge^2 (J^1 L)^\ast \otimes L)$ defines both a skew-symmetric bracket $\{-,-\}_J$ on $\Gamma (L)$ and a skew-symmetric bracket $[-,-]_J$ on $\Gamma (J^1 L)$ via
$$
\{ \lambda, \mu\}_J := J (j^1 \lambda, j^1 \mu) \quad {\rm and} \quad 
 [\phi, \psi]_J  := \Ll_{J^\sharp \phi} \psi - \Ll_{J^\sharp \psi} \phi - d_{D L}  J(\phi, \psi).$$
It is easy to see that $(L, \{-,-\}_J)$ is a Jacobi bundle {(see, e.g.~\cite{M1991})} if and only if  $(J^1L, [-,-]_J, \sigma J^\sharp)$ is a Lie algebroid (see, e.g.~\cite{CS2015,LOTV2014}), in this case we say that $J$ is a \emph{Jacobi structure on $L$}.

\begin{proposition}\label{prop:formulas}
Let $\Ii$ be a generalized almost contact structure on $L$. It is integrable if and only if, for all $\sigma, \tau \in \Gamma (J^1 L)$, $\Delta, \nabla, \square \in \Gamma (D L)$,
\begin{align}
J^\sharp [\phi, \psi]_J &= [J^\sharp \phi, J^\sharp \psi];   \label{eq:C1}\\
\varphi^\dag [\phi, \psi]_J &= \Ll_{J^\sharp \phi} \varphi^\dag \psi - \Ll_{J^\sharp \psi} \varphi^\dag \phi - d_{D L}  J (\varphi \phi, \psi); \label{eq:C2}\\
\Nn_\varphi (\Delta, \nabla) &= J^\sharp (i_\Delta i_\nabla d_{D L}  \omega); \label{eq:C3} 
\end{align}
and
\begin{equation}\label{eq:C4}
(d_{D L}  \omega_\varphi)(\Delta, \nabla, \square) = (d_{D L}  \omega) (\varphi\Delta, \nabla, \square) + (d_{D L}  \omega )(\Delta, \varphi\nabla, \square) + (d_{D L}  \omega) (\Delta, \nabla, \varphi\square), 
\end{equation}

\noindent where $\Nn_\varphi (\Delta, \nabla) := [\varphi \Delta, \varphi \nabla] + \varphi^2 [\Delta, \nabla] - \varphi[\varphi \Delta, \nabla] - \varphi [\Delta, \varphi \nabla]$.
\end{proposition}

Equations (\ref{eq:alg}) and (\ref{eq:C1})-(\ref{eq:C4}) should be seen as \emph{structure equations of a generalized contact structure}. 

\begin{example}
Let $\Ii$ be a generalized almost contact structure on $L \to M$ given by (\ref{eq:split}). As for generalized complex structures there are two extreme cases. the first one is when $\varphi = 0$, hence $J^\sharp = (\omega^\flat)^{-1}$, and $\Ii$ is completely determined by $\omega$ which is a non-degenerate Atiyah $2$-form. Now, $\Ii$ is integrable if and only if $d_{D L}  \omega = 0$, hence $\omega$ corresponds to a contact structure $H$ on $M$ such that $TM/H = L$. The second extreme case is when $J = \omega = 0$, hence $\varphi^2 = - \operatorname{id}$, i.e.~$\varphi$ is an almost complex structure on the Atiyah algebroid $D L$, and $\Ii$ is integrable if and only if $\varphi$ is a complex structure \cite{BR2010}.
\end{example}

Equation (\ref{eq:C1}) says that $J$ is a Jacobi structure. So every generalized contact bundle has an underlying Jacobi structure. 
Equation (\ref{eq:C2}) describes a compatibility condition between $J$ and $\varphi$. Equation (\ref{eq:C3})  measures the non-integrability of $\varphi$ while (\ref{eq:C4})  is a compatibility condition between $\varphi$ and $\omega$.

\begin{remark}
Consider the manifold $\widetilde M := L^\ast \smallsetminus \mathbf 0$ ($\mathbf 0$ being the image of the zero section). Recall that $\widetilde M$ is a principal $\mathbb R^\times$-bundle (and every principal $\mathbb R^\times$-bundle arise in this way). In particular $\widetilde M$ comes equipped with an homogeneity structure $h : [0, \infty)  \times M \to M$ in the sense of Grabowski (see, e.g.~\cite{BGG2015}). The fundamental vector field corresponding to the canonical generator $1$ in the Lie algebra $\mathbb R$ of $\mathbb R^\times$ will be denoted by $\mathcal E$.
\begin{proposition}\label{claim:hom_gcs}
Generalized contact structures on $L$ are in one-to-one correspondence with \emph{homogeneous generalized complex structures} on $\widetilde M$.
\end{proposition}
Let $\Jj : \mathbb T \widetilde M \to \mathbb T \widetilde M$ be a generalized complex structures. Using decomposition $\mathbb T \widetilde M = T \widetilde M \oplus T^\ast \widetilde M$ we see that $\Jj$ is the same as a triple $(a, \pi, \sigma)$ where $a$ is an endomorphism of $T \widetilde M$, $\pi$ is a bi-vector field, and $\sigma$ is a $2$-form on $\widetilde M$ satisfying suitable identities \cite{C2011}. We say that $\Jj$ is \emph{homogeneous} if 1) $a$ is homogeneous of degree $0$, i.e.~$\mathcal L_{\mathcal E} a = 0$, 2) $\pi$ is homogeneous of degree $-1$, i.e.~$\mathcal L_{\mathcal E} \pi = - \pi$, and 3) $\omega$ is homogeneous of degree $1$, i.e.~$\mathcal L_{\mathcal E} \sigma = \sigma$. Now, Proposition \ref{claim:hom_gcs} is a straightforward consequence of \cite[Theorem A.4]{V2015}. There is an alternative elegant way of  explaining the homogeneity of $\Jj$. Namely, homogeneity structure $h$ lifts to an homogeneity structure $h_{\mathbb T}$  on the generalized tangent bundle, namely the direct sum of the \emph{tangent} and the \emph{phase lifts} of $h$ \cite[Section 2]{BGG2015}. It is easy to check that $\Jj$ is homogeneous in the above sense iff it is equivariant with respect to $h_{\mathbb T}$ (see also \cite[Theorem 2.3]{BGG2015}). 
\end{remark}

It is useful to characterize those generalized contact structures such that $J$ is non-degenerate. In this case there is a unique non-degenerate Atiyah $2$-form $\omega_J$, also denoted $J^{-1}$, such that $J^\sharp \omega_J^\flat = \operatorname{id}$ and (\ref{eq:C1}) says that $\omega_J$ is $d_{D L} $-closed. Hence it comes from a contact structure $H_J \subset TM$ such that $TM/H_J = L$. Following Crainic \cite{C2011} we introduce the following notion:

\begin{definition}
A \emph{contact-Hitchin pair} on a line bundle $L \to M$ is a pair $(H, \Phi)$ consisting of a contact structure $H \subset TM$ with $TM/H = L$, and an endomorphism $\Phi : D L \to D L$ such that
(i) $\Omega^\flat \Phi = \Phi^\dag \Omega^\flat$ (so that the Atiyah $2$-form $\Omega_\Phi$ is well-defined), and
(ii) $d_{D L}  \Omega_\Phi = 0$,
where $\Omega$ is the Atiyah $2$-form corresponding to $H$, i.e.~$\Omega := d_{D L}  \sigma^\ast \theta$, and $\theta : TM \to L$ is the structure form of $H$, i.e.~$H = \ker \theta$.
\end{definition}

\begin{proposition}\label{prop:Hitchin_pair}
There is a one-to-one correspondence between generalized contact structures on $L$ given by (\ref{eq:split}), with $J$ non-degenerate, and contact Hitchin pairs $(H, \Phi)$ on $L$. In this correspondence $H$ is the contact structure corresponding to $\omega_J = J^{-1}$, and moreover:
 $$\quad \Phi = \varphi \quad  {\rm and}  \quad \omega = - (\omega_J + \varphi^\ast \omega_J), \quad {\rm where} \quad (\varphi^\ast \omega_J) (\Delta, \nabla) := \omega_J (\varphi \Delta, \varphi\nabla).$$
\end{proposition}

The proof of Proposition \ref{prop:Hitchin_pair} is similar to that of Proposition 2.6 in \cite{C2011}.

\section{Multiplicative Atiyah forms on Lie groupoids and generalized contact structures}

{As we have seen above,} every generalized contact bundle $(L, \Ii)$ has an underlying Jacobi structure $J$. Jacobi structures are the infinitesimal counterparts of multiplicative contact structures on Lie groupoids \cite{CS2015} (see also \cite{KS1993} for the case $L = M \times \R$). So, it is natural to ask: \emph{are the remaining components $\varphi, \omega$ of $\Ii$ also infinitesimal counterparts of suitable (multiplicative) structures on $\Gg$?} Theorem \ref{theor:Hitchin1} below answers this question (see \cite[Theorems 3.2, 3.3, 3.4]{C2011} for the generalized complex case). In order to state it, we need to introduce a novel notion of \emph{multiplicative Atiyah form}.

Multiplicative forms and their infinitesimal counterparts are extensively studied in \cite{BC2012}. Vector-bundle valued differential forms and their infinitesimal counterparts, Spencer operators, are studied in \cite{CSS2015}. 
In what follows, we outline a similar theory for Atiyah forms. 

Given a Lie groupoid $\Gg \rightrightarrows M$, we will denote the source by $s$, the target by $t$ and the unit by $u$. Moreover, we identify $M$ with its image under $u$. Denote by $\Gg_2 = \{ (g_1, g_2) \in \Gg \times \Gg: s(g_1) = t (g_2)\}$ the manifold of composable arrows and let $m : \Gg_2 \to \Gg$, $(g_1, g_2) \mapsto g_1g_2$ be the multiplication. We denote by $\pr_1, \pr_2 : \Gg_2 \to \Gg$ the projections onto the first and second factor respectively.

Recall that the Lie algebroid $A$ of $\Gg$ consists of tangent vectors to the source fibers at points of $M$. Every section $\alpha$ of $A$ corresponds to a unique right invariant, $s$-vertical vector field $\alpha^r$ on $\Gg$ such that $\alpha = \alpha^r |_M$. Now let $E \to M$ be a vector bundle carrying a representation of $\Gg$. {Thus,  there is a flat $A$-connection $\nabla$ in $E$.} As shown in \cite[Proposition 10.1]{V2015}, there is a canonical flat $\ker ds$-connection $\nabla^{\Gg} : \ker ds \to D (t^\ast E)$ in the pull-back bundle $t^\ast E$ such that
$
\nabla_\alpha = t_\ast (\nabla^\Gg_{\alpha^r}|_M)
$
for all $\alpha \in \Gamma (A)$.
Additionally, there is a natural vector bundle isomorphism (covering the identity)
$
\mathsf{i} : (t \circ \pr_1)^\ast E \longrightarrow (t \circ \pr_2)^\ast E
$
defined as follows. For $((g_1,g_2), e) \in (t \circ \pr_1)^\ast E$, $e \in E_{t(g_1)}$ put
$
\mathsf{i} ((g_1,g_2), e) := ((g_1,g_2), g_1^{-1} \cdot e) \in (t \circ \pr_2)^\ast E.
$
In the following $E = L$ is a line bundle. In particular, $(t \circ \pr_2)^\ast L$-valued Atiyah forms can be pulled-back to $(t \circ \pr_1)^\ast L$-valued Atiyah forms along $\mathsf{i}$.

\begin{definition}
An Atiyah form $\omega \in \Omega^\bullet_{t^\ast L}$ is \emph{multiplicative} if
$
m^\ast \omega = \pr_1^\ast \omega + \mathsf{i}^\ast \pr_2^\ast \omega.
$
\end{definition}

We also need the following:

\begin{definition}
An endomorphism $\Phi : D (t^\ast L) \to D (t^\ast L)$ is \emph{multiplicative} if, for every $\square \in D (m^\ast t^\ast L)$, there is a, necessarily unique, $\Phi D \in D (m^\ast t^\ast L)$ such that (1) $\pr_{1\ast} \Phi \square = \Phi \pr_{1\ast} \square$, (2) $\pr_{2\ast} \mathsf{i}_\ast \Phi \square = \Phi \pr_{2\ast} \mathsf{i}_\ast \square$ and (3) $m_\ast \Phi \square = \Phi m_\ast \square$.
\end{definition}

The following definition and Theorem \ref{theor:Hitchin1} provide the infinitesimal counterpart of multiplicative Atiyah forms. Let $L \to M$ be a line bundle carrying a representation of a Lie algebroid $A \to M$. 

\begin{definition}\label{def:mult_Atiyah}
An \emph{$L$-valued infinitesimal multiplicative (IM) Atiyah $k$-form on $A$} is a pair $(\boldsymbol \Dd, \boldsymbol l)$, where 
$
\boldsymbol \Dd : \Gamma (A) \longrightarrow \Omega^k_L
$
is a first order differential operator, and 
$
\boldsymbol l : A \longrightarrow \wedge^{k-1} (D L)^\ast \otimes L
$
is a vector bundle morphism such that, for all $\alpha, \beta \in \Gamma (A)$, and $f \in C^\infty (M)$, 
\[
\boldsymbol \Dd (f \alpha)  = f \boldsymbol \Dd (\alpha) + d_{D L}  f \wedge \boldsymbol l (\alpha),
\]
and
\[
\begin{aligned}
  \Ll_{\nabla_\alpha} \boldsymbol \Dd(\beta) - \Ll_{\nabla_\beta} \boldsymbol \Dd(\alpha) & = \boldsymbol \Dd([\alpha, \beta]) ;   \\
\Ll_{\nabla_\alpha} \boldsymbol l (\beta) - i_{\nabla_\beta} \boldsymbol \Dd(\alpha)  &  = \boldsymbol l ([\alpha, \beta]) ;  \\
i_{\nabla_\alpha} \boldsymbol l (\beta) + i_{\nabla_\beta} \boldsymbol l (\alpha)& =0.
 \end{aligned}
 \]
\end{definition}

Now on, $L \to M$ is a line bundle carrying a representation of a source simply connected Lie groupoid $\mathcal G \rightrightarrows M$, and $A$ is the Lie algebroid of $\mathcal G$.

\begin{theorem}\label{theor:Hitchin1_bis}
There is a one-to-one correspondence between $t^\ast L$-valued multiplicative Atiyah $k$-forms $\omega$ and $L$-valued IM Atiyah $k$-forms $(\boldsymbol \Dd, \boldsymbol l)$ on $A$. In this correspondence
\begin{equation*}
\boldsymbol \Dd(\alpha) =  u^\ast (\Ll_{\nabla^\Gg_{\alpha^r}} \omega)\quad {\rm and} \quad 
\boldsymbol l (\alpha)  = u^\ast (i_{\nabla^\Gg_{\alpha^r}} \omega) .
\end{equation*}
\end{theorem}

\begin{proof}
There is a direct sum decomposition $\Omega^\bullet_{t^\ast L} = \Omega^\bullet (\Gg, t^\ast L) \oplus \Omega^\bullet (\Gg, t^\ast L)[1]$ given by $\omega \equiv (\omega_0, \omega_1)$, with $\omega = \sigma^\ast \omega_0 + d_{DL} \sigma^\ast \omega_1$, and $\omega$ is multiplicative if and only if $(\omega_0, \omega_1)$ are so. Using \cite[Theorem 1]{CSS2015}, we see that $(\omega_0, \omega_1)$ correspond to Spencer operators on $A$ \cite[Definition 2.6]{CSS2015}. Finally check that, similarly as for Atiyah forms, IM Atiyah forms decompose canonically into a direct sum of Spencer operators.
\end{proof}

\begin{remark}\label{rem:Hitchin}
Let $H \subset T\Gg$ be a multiplicative contact structure on $\Gg$, with $T\Gg/ H = t^\ast L$ and let $\Omega$ be the corresponding Atiyah $2$-form. When specialized to $\Omega$, Theorem \ref{theor:Hitchin1_bis} gives an isomorphism $\boldsymbol l : A \to J^1 L$ of $A$ with the Lie algebroid $(J^1 L, [-,-]_J, \sigma J^\sharp)$ corresponding to a unique Jacobi structure $J$ on $L$.
\end{remark}

\begin{definition}
A \emph{contact-Hitchin groupoid} is a Lie groupoid $\Gg \rightrightarrows M$ together with 

\noindent (1) a line bundle $L \to M$ carrying a representation of $\Gg$,

\noindent (2) a multiplicative contact-Hitchin pair $(H, \Phi)$ on $t^\ast L$, i.e.~both $H$ and $\Phi$ are multiplicative, and

\noindent (3) an $L$-valued Atiyah $2$-form $\omega$ on $M$ such that $\Omega + \Phi^\ast \Omega = s^\ast \omega - t^\ast \omega$,

\noindent where $\Omega$ is the Atiyah $2$-form corresponding to $H$.
\end{definition}

\begin{theorem}\label{theor:Hitchin1}
There is a one-to-one correspondence between contact-Hitchin groupoid structures $(H, \Phi, \Omega)$ on $\Gg$ and triples $(J, \varphi, \omega)$ satisfying Equations (\ref{eq:C1})-(\ref{eq:C3}),  and the first two equations in (\ref{eq:alg}). In this correspondence, $J$ is the Jacobi structure corresponding to $H$ (Remark \ref{rem:Hitchin}), and $\varphi : DL \to DL$ is the (well-defined) restriction of $\Phi$ to $D L$.
\end{theorem}
 
\noindent  Theorem \ref{theor:Hitchin1} can be proved using arguments similar to those in Theorems 3.3 and 3.4 in \cite{C2011}. Alternatively, one could use a conceptual approach similar to that of \cite{SX2007}, exploiting the notion of Jacobi quasi-Nijenhuis structure \cite{S2010,CNN2010}. Finally, we observe that the last equation of (\ref{eq:alg}) and Equation (\ref{eq:C4}) do not have a Lie groupoid/Lie algebroid interpretation.

\end{document}